\documentclass[oneside,11pt,reqno]{amsart}

\usepackage[margin=1in]{geometry}

\usepackage{cellspace}
\usepackage{caption}
\usepackage{adjustbox}
\usepackage{float}
\usepackage{amssymb}
\usepackage{latexsym}
\usepackage[mathcal]{euscript}
\usepackage{hyperref}
\usepackage{xcolor}
\usepackage{ytableau}
\usepackage{tikz}
\usetikzlibrary{positioning}
\ytableausetup{centertableaux}

\newtheorem{theorem}{Theorem}[section]
\newtheorem{corollary}[theorem]{Corollary}
\newtheorem{lemma}[theorem]{Lemma}
\newtheorem{proposition}[theorem]{Proposition}
\theoremstyle{definition}

\newtheorem{conjecture}[theorem]{Conjecture}
\newtheorem{question}[theorem]{Question}

\theoremstyle{remark}
\newtheorem{remark}[theorem]{Remark}

\numberwithin{equation}{section}
\numberwithin{figure}{section}
\numberwithin{table}{section}
\newcommand{\Ev}{\mathrm{Ev}(\lambda)}

\newcommand{\lrb}[3]{
  \put(#2,#3){
    \begin{picture}(10,10)(0,0)
      \put(0,0){\line(1,0){10}}
      \put(0,0){\line(0,1){10}}
      \put(10,0){\line(0,1){10}}
      \put(2.5,1.3){#1}
    \end{picture}
  }
}

\newcommand{\lrt}[3]{
  \put(#2,#3){
    \begin{picture}(10,10)(0,0)
      \put(0,10){\line(1,0){10}}
      \put(0,0){\line(0,1){10}}
      \put(10,0){\line(0,1){10}}
      \put(2.5,1.3){#1}
    \end{picture}
  }
}

\newcommand{\lrbt}[3]{
  \put(#2,#3){
    \begin{picture}(10,10)(0,0)
      \put(0,10){\line(1,0){10}}
      \put(10,0){\line(0,1){10}}
      \put(0,0){\line(1,0){10}}
      \put(0,0){\line(0,1){10}}
      \put(2.5,1.3){#1}
    \end{picture}
  }
}

\begin{document}

\begin{abstract}
  Amdeberhan recently proposed certain equalities between sums in the character table of symmetric groups.
  These equalities are between signed column sums in the character table, summing over the rows labeled by partitions in $\Ev$,
  where $\lambda$ is a partition of $n$ with $r$ nonzero parts and $\Ev$ is a multiset containing $2^r$ partitions of $2n$.
  While we observe that these equalities are not true in general, we prove that they do hold in interesting special cases. These lead to new equalities between sums of degrees of irreducible characters for the symmetric group and a new combinatorial interpretation for the Riordan numbers in terms of degrees of irreducible characters labeled by partitions with three parts of the same parity. This is the first, to our knowledge, theorem about degrees of symmetric group characters with parity conditions imposed on the partitions indexing the characters.
\end{abstract}

\title{New Identities in the Character Table of Symmetric Groups involving Riordan Numbers}

\author{David J. Hemmer}
\address{Department of Mathematical Sciences\\
  Michigan Technological University\\
  Houghton, MI 49931}
\email{djhemmer@mtu.edu}

\author{Armin Straub}
\address{Department of Mathematics and Statistics\\
  University of South Alabama\\
  Mobile, AL 36688}
\email{straub@southalabama.edu}

\author{Karlee J. Westrem}
\address{Department of Mathematical Sciences\\
  Appalachian State University\\
  Boone, NC 28608}
\email{westremk@appstate.edu}

\maketitle

\section{Introduction}
\label{section: Introduction}

Let $\Sigma_n$ denote the symmetric group on $n$ letters. Recall that both the complex irreducible characters and the conjugacy classes of $\Sigma_n$ are indexed by partitions of $n$. For two partitions $\lambda$ and $\mu$ of $n$, we let $\chi^\mu_\lambda$ be the value of the irreducible character $\chi^\mu$ on a permutation of cycle type $\lambda$. The degree $\chi^\mu_{(1^n)} = \chi^\mu(1)$ is the number of standard Young tableaux of shape $\mu$, which we denote by $f^\mu$. A good reference for the complex representation theory of $\Sigma_n$ is \cite{Sagan}.

We denote the size of a partition $\lambda$ and its length by $|\lambda|$ and $\ell(\lambda)$, respectively. Let $\lambda'$ denote the conjugate partition obtained by reflecting the Young diagram of $\lambda$ across the  main diagonal.

For the partition $\tau = (1^n)$, we have that $\chi^\tau$ is the linear character corresponding to the signature representation, denoted $\operatorname{sgn}$. Recall that
\begin{equation}
  \chi^\mu \otimes \operatorname{sgn} = \chi^{\mu'}.
  \label{eq: tensorwithsign}
\end{equation}

Let $\lambda = (\lambda_1, \lambda_2, \dots, \lambda_r)$ be a partition of $n$ with $r$ nonzero parts, so $\ell(\lambda) = r$. In \cite{AmdeberhanOverflow}, Amdeberhan defined $\Ev$ to be the set of all partitions of $2n$ obtained by replacing each $\lambda_i$ with either $2\lambda_i$ or two copies of $\lambda_i$, and then reordering the parts to be nonincreasing. Thus $\Ev$ is a multiset containing $2^{\ell(\lambda)}$ partitions of $2n$. For example, let $\lambda = (3,2,2)$. Then:
\begin{equation}
  \Ev = \{(6,4^2), (6,4,2^2), (6,4,2^2), (6,2^4), (4^2,3^2), (4,3^2,2^2), (4,3^2,2^2), (3^2,2^4)\}.
  \label{eq: Example Ev}
\end{equation}

Observe that products of binomial coefficients may arise in the multiplicities of $\Ev$. For example, for $\lambda = (3^4,2^3)$, we will get $\binom{4}{2}\binom{3}{1}$ copies of $(6^2,3^4,4,2^4)$ in $\Ev$.

Amdeberhan defines two subsets of partitions of size $2n$ by requiring either the rows or the columns be of even length and by restricting the number of parts:
\begin{equation}
  \begin{aligned}
    \mathcal{R}_N(2n) &:= \{ \mu \vdash 2n \mid \ell(\mu) \leq N, \mu_i \text{ is even for all } i \}, \\
    \mathcal{R}_N^c(2n) &:= \{ \mu \vdash 2n \mid \ell(\mu) \leq N, \mu_i' \text{ is even for all } i \}.
  \end{aligned}
  \label{eq: DefineRNRNc}
\end{equation}

For example,
\begin{equation}
  \begin{aligned}
    \mathcal{R}_3(10) &:= \{ (10), (8,2), (6,4), (6,2,2), (4,4,2) \}, \\
    \mathcal{R}_4^c(10) &:= \{ (5,5), (4,4,1,1), (3,3,2,2) \}.
  \end{aligned}
  \label{eq: exampleRNRNc}
\end{equation}

\begin{remark}
  \label{remark: RN}
  Notice that $\mathcal{R}_2^c(2n)$ is the single partition $(n,n)$. Also notice that once $N \geq n$, the set $\mathcal{R}_N(2n)$ stabilizes, and once $N \geq 2n$, the set $\mathcal{R}_N^c(2n)$ stabilizes. Thus, for $N \geq 2n$, the partitions in $\mathcal{R}_N(2n)$ are precisely the conjugates of the partitions in $\mathcal{R}_N^c(2n)$.
\end{remark}

\subsection{Motzkin and Riordan Numbers}
The \emph{Motzkin numbers} $M(n)$ are the sequence A001006 in \cite{oeis}. $M(n)$ counts the number of \emph{Motzkin paths}, which are paths from $(0,0)$ to $(n,0)$ using only steps $U = (1,1)$, $F = (1,0)$, and $D = (1,-1)$, and not going below the $x$-axis. $M(n)$ is also well known to count the number of standard Young tableaux of size $n$ with three or fewer rows; for an example of a bijection, see \cite{MatsakisMotzkinInspired}. These are also easily seen to be in bijection with three-candidate ballot sequences of length $n$. A three-candidate ballot sequence of length $n$ is a string $(b_1, b_2, \dots, b_n)$ with $b_i \in \{A,B,C\}$ such that each prefix $b_1, b_2, \dots, b_k$, with $1 \leq k \leq n$, has the property that the number of $A$'s is no less than the number of $B$'s, which in turn is no less than the number of $C$'s. Here, the three candidates are represented by $A$, $B$, and $C$, and each $b_j$ represents a vote for one of the candidates. 

A \emph{Riordan path} is a Motzkin path with the additional requirement that there may not be a flat step $F$ on the $x$-axis, i.e., an $F$ may only appear if there are more $U$s than $D$'s prior to it in the sequence. We let $R(n)$ be the number of Riordan paths of length $n$; this is the sequence A005043 in \cite{oeis}. It is well-known that:
\begin{equation}
  M(n) = R(n) + R(n+1).
  \label{eq: M=R+R}
\end{equation}

The number $R(n)$ also has an interpretation in terms of degrees of irreducible symmetric group characters. This is a comment of Regev given on the OEIS entry without proof, so we present a proof of a refinement here:

\begin{proposition}
  \label{prop:RiordannumberSYT}
  Let $0 \leq m < n$. The number of Riordan paths of length $n$ with $m$ flat steps and $k$ up steps (and thus $k$ down steps) is $f^{(k,k,1^m)}$.
\end{proposition}

\begin{proof}
  Given a standard tableau of shape $(k,k,1^m)$, we describe how to construct a corresponding Riordan path. The numbers in the first row correspond to the positions of the $U$'s. With the first row determined, the first entry in row 2 is forced; we can think of it as corresponding to the $D$ at the end of the sequence. Now there are $k+m-1$ numbers left in the tableau and $k+m-1$ empty positions in the sequence. List the remaining entries of the tableau in order, marking the $m$ entries from the first column as $F$ and the remaining $k-1$ entries from the second row as $D$. Then simply fill in the empty positions with these entries in the same relative order. For example, let $k=m=3$ and
  \begin{equation}
    T = \begin{ytableau}
      1 & 2 & 7 \\
      3 & 5 & 8 \\
      4 \\
      6 \\
      9
    \end{ytableau}.
    \label{eq: Tableau example}
  \end{equation}

  Filling in $U$s in positions $1,2,7$ and $D$ in the last position, we obtain $UU????U?D$. The remaining entries are $\{{\color{red} 4}, 5, {\color{red} 6}, 8, {\color{red} 9}\}$, where red denotes the entries in the final $m$ rows. Thus, we get a corresponding sequence $FDFDF$, which we use to replace the question marks, arriving at $UUFDFDUFD$. It is straightforward to verify the resulting sequence is Riordan and to invert the map.
\end{proof}

Summing over all $m$, an immediate corollary is the observation of Regev:
\begin{corollary}
  \label{cor: Riordan sum}
  The number of Riordan paths of length $n$ is:
  \begin{equation}
    R(n) = \sum_{k=1}^{\lfloor n/2 \rfloor} f^{(k,k,1^{n-2k})}.
    \label{eq: Riordan sum}
  \end{equation}
\end{corollary}

\section{Amdeberhan's Question}
\label{sec:amdeberhan:question}

For a partition $\lambda$ with part frequencies $(1^{m_1}, 2^{m_2}, \dots, n^{m_n})$, we denote with
$$z_\lambda = \prod_{j=1}^n j^{m_j} m_j!$$
the size of the centralizer of a permutation with cycle type $\lambda$.
In a 2023 post on MathOverflow \cite{AmdeberhanOverflow}, Amdeberhan raised the following intriguing question about certain weighted sums of irreducible characters for symmetric groups, with columns corresponding to $\mathcal{R}_N(2m)$ or $\mathcal{R}_N^c(2m)$ and rows corresponding to $\Ev$.

\begin{question}[\cite{AmdeberhanOverflow}]
  \label{q:amdeberhan}
  Is it true that
  \begin{equation}
    \sum_{\lambda \vdash n} \frac{1}{z_\lambda} \sum_{\tilde{\lambda} \in \Ev} \sum_{\mu \in \mathcal{R}_{2N+1}(2n)} (-1)^{\ell(\tilde{\lambda})}  \chi^\mu_{\tilde{\lambda}} =
    \sum_{\lambda \vdash n} \frac{1}{z_\lambda} \sum_{\tilde{\lambda} \in \Ev} \sum_{\mu \in \mathcal{R}_{2N}^c(2n)} \chi^\mu_{\tilde{\lambda}}
    \label{eq: q:amdeberhan}
  \end{equation}
  for given integers $n, N \geq 1$?
\end{question}

In \cite{AmdeberhanOverflow}, it is noted that the answer to that question is affirmative for sufficiently large $N$.  More precisely, we observe that \eqref{eq: q:amdeberhan} is true if $N \ge n$. This follows from Remark~\ref{remark: RN}, by which the sets $\mathcal{R}_{2N+1}(2n)$ and $\mathcal{R}^c_{2N}(2n)$ contain exactly conjugate partitions if $N \ge n$, combined with the following fact: suppose $\sigma \in \Sigma_{2n}$ has cycle type $\tilde{\lambda}$. Since $2n$ is even, $(-1)^{\ell(\tilde{\lambda})} = 1$ precisely when $\sigma$ is an even permutation. Thus, by \eqref{eq: tensorwithsign}, for any $\mu \vdash 2n$, we have:
\begin{equation}
  (-1)^{\ell(\tilde{\lambda})} \chi^\mu_{\tilde{\lambda}} = \chi^{\mu'}_{\tilde{\lambda}}.
  \label{eq:removealternatingsign}
\end{equation}

In \cite{AmdeberhanOverflow}, it was asked whether \eqref{eq: q:amdeberhan} is true for all $n, N \geq 1$.  By explicitly computing both sides of \eqref{eq: q:amdeberhan} for all $N < n$ and small fixed $n$, we find that \eqref{eq: q:amdeberhan} is true for all $n \le 11$. However, the identity does not continue to hold for $n=12$. In the case $n = 12$, \eqref{eq: q:amdeberhan} holds for $N = 1$ and $N = 2$ but then for fails for $N = 3$, where the left-hand side evaluates to $1040$ and the right-hand side to $1041$.

On the other hand, our computations of initial cases show that \eqref{eq: q:amdeberhan} is true for $N=1$ and $N=2$ in all cases $n \le 15$. In those cases, the inner double-sums in identity \eqref{eq: q:amdeberhan} match for all partitions $\lambda$. We therefore also consider the following stronger version of Amdeberhan's question.

\begin{question}
  \label{q:amdeberhan:x}
  For which partitions $\lambda \vdash n$ and which integers $N \geq 1$ does the following identity hold?
  \begin{equation}
    \sum_{\tilde{\lambda} \in \Ev} \sum_{\mu \in \mathcal{R}_{2N+1}(2n)} (-1)^{\ell(\tilde{\lambda})} \chi^\mu_{\tilde{\lambda}} =
    \sum_{\tilde{\lambda} \in \Ev} \sum_{\mu \in \mathcal{R}_{2N}^c(2n)} \chi^\mu_{\tilde{\lambda}}
    \label{eq: stronger form}
  \end{equation}
\end{question}

As observed above for \eqref{eq: q:amdeberhan}, this identity holds for all partitions $\lambda \vdash n$ in the case that $N \ge n$.
Computing all instances, we again find that the identity \eqref{eq: stronger form} holds for all partitions of size $n \le 7$ and all values of $N$.
On the other hand, \eqref{eq: stronger form} does not hold for certain partitions of size $n=8$ if $N=3$. We provide some more details on this case in Appendix~\ref{sec:appendix-counterexamples}.
However, we conjecture that \eqref{eq: stronger form} is true for all partitions $\lambda$ if $N=1$.
When $N=1$, the set $\mathcal{R}_{2N}^c(2n)$ is the single partition $(n,n)$, and the set $\mathcal{R}_{2N+1}(2n)$ consists of all partitions of $2n$ with at most three parts, all even.
\begin{conjecture}[$N=1$ version]
  \label{conj:N=1 version}
  For any partition $\lambda \vdash n$, we have
  \begin{equation}
    \sum_{\tilde{\lambda} \in \Ev} \sum_{\mu \in \mathcal{R}_3(2n)} (-1)^{\ell(\tilde{\lambda})} \chi^\mu_{\tilde{\lambda}} =
    \sum_{\tilde{\lambda} \in \Ev} \chi^{(n,n)}_{\tilde{\lambda}}.
    \label{eq: N=1form}
  \end{equation}
\end{conjecture}

In this paper, we will prove this conjecture in the case where $\lambda = (c^n)$. We find fundamentally different behavior for the $c=1$ case and the $c>1$ case.

\subsection{Two examples}
We will prove Conjecture~\ref{conj:N=1 version} for the case $\lambda = (c^n)$.
We illustrate here this result, which looks different for $\lambda=(1^n)$ and $\lambda=(c^n)$ for $c>1.$ Let $n=4$ and assume first that $\lambda=(1^4).$ Then the multiset $\Ev$ has $2^4=16$ elements: $$\Ev = \{(1^8),(2,1^6),(2^2,1^4),(2^3,1^2),(2^4)\}$$ with multiplicities $1,4,6,4,1$, respectively. We are assuming $N=1$ so 

$$\mathcal{R}_{2N+1}(8)=\mathcal{R}_{3}(8) =\{(8),(6,2),(4,4),(4,2,2)\}$$ and 
$$\mathcal{R}_{2N}^c(8)=\mathcal{R}_{2}^c(8)=\{(4,4)\}.$$

\begin{table}[h]

\centering

\small 
\begin{adjustbox}{max width=\textwidth} 
\begin{tabular}{c|cccccc}
\hline
$\mu \setminus \tilde{\lambda}$ & $[1^8]$ & $[2,1^6]$ & $[2,2,1^4]$ & $[2,2,2,1^2]$ & $[2,2,2,2]$ & \\
\hline
$[8]$ & 1 & 1 & 1 & 1 & 1 \\
$[6,2]$ & 20 & 10 & 4 & 2 & 4  \\
$[4,4]$ & 14 & 4 & 2 & 0 & 6  \\
$[4,2,2]$ & 56 & 4 & 0 & 4 & 8  \\
&&&&&&\\
Column Sum &91 &19 &7 &7 & 19 & \\
Weight & 1&  $-4$& 6& $-4$ & 1\\
Total&91&$-76$&42&$-28$&19\\
\hline
\end{tabular}
\end{adjustbox}
\captionsetup{aboveskip=0.5cm} 
\caption{Partial character table of $\Sigma_8$ for the case $\lambda=(1^4)$}

\label{table:S8lambda=1111}
\end{table}

Table \ref{table:S8lambda=1111} is a partial character table of $\Sigma_8$. The columns correspond to partitions $\tilde{\lambda}$ in $\Ev$ and the row ``weight'' is just $(-1)^{\ell(\tilde{\lambda})}$ times the multiplicity of $\tilde{\lambda}$. Summing the totals we obtain
$$91 \cdot 1 -19 \cdot 4 +7 \cdot 6-7 \cdot 4+19 \cdot 1=48=2^4 \cdot 3$$
as the LHS of Conjecture~\ref{conj:N=1 version}. For the RHS we look only at the row corresponding to $(4,4)$ and compute a weighted sum without the signs, obtaining
$$14\cdot 1 + 4 \cdot 4 +2 \cdot 6+ 0 \cdot 4 + 6 \cdot 1=48=2^4 \cdot 3.$$
We will prove that the $3$'s on the right-hand sides are the Riordan number $R(4)$ and also equal to the sums $f^{(1,1,1,1)}+f^{(2,2)} = f^{(4)}+f^{(2,2)}$ (see Theorem~\ref{thm:equalitySpecht}).

For an example illustrating the case $c>1$, consider $n=4$ again and now choose $\lambda=(2,2)$. Then:
$$\Ev=\{(2^4),(4,2^2),(4^2)\}$$
with multiplicities 1,2,1, respectively.

\begin{table}[h]
\centering
\small 
\begin{adjustbox}{max width=\textwidth} 
\begin{tabular}{c|ccc}
\hline
$\mu \setminus \tilde{\lambda}$ & $[2,2,2,2]$ & $[4,2,2]$ & $[4,4]$ \\
\hline
$[8]$ & 1 & 1 & 1 \\
$[6,2]$ & 4 & 2 & 0 \\
$[4,4]$ & 6 & 2 & 2 \\
$[4,2,2]$ & 8 & 0 & 0 \\
&&&\\
Column Sum&19&5&3\\
Weight&1&$-2$&1\\
\hline
\end{tabular}
\end{adjustbox}
\caption{Partial character table of $\Sigma_8$ for the case $\lambda=(2^2)$}
\label{table: lambda=22}
\end{table}
Looking at Table \ref{table: lambda=22}, our weighted sum of the column sums is $$1\cdot 19 -2 \cdot 5 +1 \cdot 3=12=2^2 \cdot 3,$$
which agrees with the unsigned weighted sum across the row $(4,4)$, which is $$1\cdot 6 +2 \cdot 2 +1 \cdot 2=12=2^2\cdot 3.$$

In this case we will prove that the $3$'s on the right-hand sides represent the central trinomial coefficient $T(2)$.

\section{Symmetric Functions}
In order to understand the alternating sum of character values over this unusual set $\Ev$, we will use symmetric functions and a recent identity proved by the third author. A good reference for symmetric functions is Chapter 7 of \cite{StanleyEC2}. We let $m_\lambda$, $p_\lambda$, and $s_\lambda$ denote the usual monomial, power sum, and Schur symmetric functions, respectively. There is a standard inner product on the space of symmetric functions of degree $n$ for which the Schur functions $\{ s_\lambda \mid \lambda \vdash n \}$ form an orthonormal basis. The character table of $\Sigma_n$ gives the change of basis matrix expressing the power sum basis in terms of the Schur basis. That is, we have:
\begin{lemma}\cite[Corollary 7.17.4]{StanleyEC2}
  \label{lem:char-inner-product}
  Let $\mu, \lambda \vdash n$. We have:
  \begin{equation}
    \chi^\mu_\lambda = \left\langle p_\lambda, s_\mu \right\rangle.
  \end{equation}
\end{lemma}

The following result is key to approaching Conjecture~\ref{conj:N=1 version} using symmetric functions:
\begin{theorem}\cite{westrem2024newsymmetricfunctionidentityPreprint}
  \label{thm: SymFunctionIdentity}
  If $\lambda = (\lambda_1, \lambda_2, \dots, \lambda_r) \vdash n$, then
  \begin{equation}
    \sum_{\tilde{\lambda} \in \Ev} (-1)^{\ell(\tilde{\lambda})} p_{\tilde{\lambda}} = 2^r \prod_{i=1}^r m_{\lambda_i \lambda_i}.
  \end{equation}
\end{theorem}

Applying this result to the left-hand side of \eqref{eq: N=1form}, we obtain:
\begin{align}
  \sum_{\tilde{\lambda} \in \Ev} \sum_{\mu \in \mathcal{R}_3(2n)} (-1)^{\ell(\tilde{\lambda})} \chi_{\tilde{\lambda}}^\mu
  &= \left\langle \sum_{\tilde{\lambda} \in \Ev} (-1)^{\ell(\tilde{\lambda})} p_{\tilde{\lambda}}, \sum_{\mu \in \mathcal{R}_3(2n)} s_\mu \right\rangle \notag \\
  &= \left\langle 2^r \prod_{i=1}^r m_{\lambda_i \lambda_i}, \sum_{\mu \in \mathcal{R}_3(2n)} s_\mu \right\rangle.
  \label{eq: LHSN=1}
\end{align}

On the right-hand side of \eqref{eq: N=1form}, we obtain:
\begin{align}
  \sum_{\tilde{\lambda} \in \Ev} \chi^{(n,n)}_{\tilde{\lambda}}
  &= \left\langle \sum_{\tilde{\lambda} \in \Ev} p_{\tilde{\lambda}}, s_{(n,n)} \right\rangle \notag \\
  &= \left\langle \sum_{\tilde{\lambda} \in \Ev} (-1)^{\ell(\tilde{\lambda})} p_{\tilde{\lambda}}, s_{(2^n)} \right\rangle \quad\text{ (by \eqref{eq:removealternatingsign})} \notag \\
  &= \left\langle 2^r \prod_{i=1}^r m_{\lambda_i \lambda_i}, s_{(2^n)} \right\rangle.
  \label{eq: RHSN=1}
\end{align}

Cancelling the $2^r$ from equations \eqref{eq: LHSN=1} and \eqref{eq: RHSN=1}, we find that Conjecture~\ref{conj:N=1 version} is equivalent to
\begin{equation}
  \left\langle \prod_{i=1}^r m_{\lambda_i \lambda_i}, \sum_{\mu \in \mathcal{R}_3(2n)} s_\mu \right\rangle = \left\langle \prod_{i=1}^r m_{\lambda_i \lambda_i}, s_{(2^n)} \right\rangle.
  \label{eq: ConjN=1reducedform}
\end{equation}

\section{Proof of Conjecture~\ref{conj:N=1 version} for \texorpdfstring{$\lambda = (1^n)$}{lambda = (1,1,\ldots)}}

In this section, we obtain combinatorial interpretations of both sides of \eqref{eq: ConjN=1reducedform} in the case where $\lambda = (1^n)$.  Both sides are counted by the \emph{Riordan numbers}, sequence A005043 in \cite{oeis}. As a consequence of the equality, we obtain a striking new interpretation of these numbers.

\begin{proposition}
  \label{prop:RHSN=1}
  The number $\left\langle m_{(1,1)}^n, s_{(2^n)} \right\rangle$ is equal to the $n$-th Riordan number $R(n)$.
\end{proposition}

\begin{proof}
  The monomial symmetric function $m_{(1,1)}$ is the same as the Schur function $s_{(1,1)}$. So the inner product in question is the Littlewood--Richardson coefficient $c_{(1,1),(1,1),\ldots,(1,1)}^{(2^n)}$. This coefficient can be determined by a simple application of the Littlewood--Richardson rule: we are counting standard tableaux of shape $(2^n)$, with entries $\{1, 2, 3, \dots, 2n\}$, such that at each step, when we add $2k+1, 2k+2$ to the existing tableau consisting of $\{1, 2, \dots, 2k\}$, the entries $\{2k+1, 2k+2\}$ are in different rows. This is the dual version of Pieri's rule; see, for instance, \cite[Sec.~7.15]{StanleyEC2}.

  Alternatively, we can think of starting from the empty tableau. At each step, we can add a domino to column one, a domino to column two if it is at least 2 boxes shorter than column one, or a single box at the end of each column if the columns are not the same length. Assigning a domino in column one to $U$, column two to $D$, and a box in both columns to $F$, we see that these tableaux are in bijection with Riordan paths.
  \end{proof}

  For example, when $n=5$, we get the six tableaux shown in Figure \ref{figure:n=5LRtableau}.
  \begin{figure}[ht]
    \centering
    \setlength{\unitlength}{0.0225in}
    \begin{picture}(160,60)(0,0)
      \lrt{1}{0}{50}
      \lrb{2}{0}{40}
      \lrt{3}{0}{30}
      \lrb{4}{0}{20}
      \lrbt{6}{0}{10}
      \lrbt{5}{10}{50}
      \lrt{7}{10}{40}
      \lrb{8}{10}{30}
      \lrt{9}{10}{20}
      \lrb{10}{10}{10}

      \put(10,0){$t_1$}
      \put(40,0){$t_2$}
      \put(70,0){$t_3$}
      \put(100,0){$t_4$}
      \put(130,0){$t_5$}
      \put(160,0){$t_6$}

      \lrt{1}{30}{50}
      \lrb{2}{30}{40}
      \lrt{5}{30}{30}
      \lrb{6}{30}{20}
      \lrbt{8}{30}{10}
      \lrt{3}{40}{50}
      \lrb{4}{40}{40}
      \lrbt{7}{40}{30}
      \lrt{9}{40}{20}
      \lrb{10}{40}{10}

      \lrt{1}{60}{50}
      \lrb{2}{60}{40}
      \lrbt{4}{60}{30}
      \lrbt{6}{60}{20}
      \lrbt{8}{60}{10}
      \lrbt{3}{70}{50}
      \lrbt{5}{70}{40}
      \lrbt{7}{70}{30}
      \lrt{9}{70}{20}
      \lrb{10}{70}{10}

      \lrt{1}{90}{50}
      \lrb{2}{90}{40}
      \lrbt{4}{90}{30}
      \lrt{7}{90}{20}
      \lrb{8}{90}{10}
      \lrbt{3}{100}{50}
      \lrt{5}{100}{40}
      \lrb{6}{100}{30}
      \lrt{9}{100}{20}
      \lrb{10}{100}{10}

      \lrt{1}{120}{50}
      \lrb{2}{120}{40}
      \lrt{3}{120}{30}
      \lrb{4}{120}{20}
      \lrbt{8}{120}{10}
      \lrt{5}{130}{50}
      \lrb{6}{130}{40}
      \lrbt{7}{130}{30}
      \lrt{9}{130}{20}
      \lrb{10}{130}{10}

      \lrt{1}{150}{50}
      \lrb{2}{150}{40}
      \lrbt{4}{150}{30}
      \lrt{5}{150}{20}
      \lrb{6}{150}{10}
      \lrbt{3}{160}{50}
      \lrt{7}{160}{40}
      \lrb{8}{160}{30}
      \lrt{9}{160}{20}
      \lrb{10}{160}{10}
    \end{picture}
    \caption{Tableaux for $n=5$ in the proof of Proposition \ref{prop:RHSN=1}}
    \label{figure:n=5LRtableau}
  \end{figure}

  The six tableaux $t_1,\ldots,t_6$ in Figure \ref{figure:n=5LRtableau} correspond to the sequences $s_1,\ldots,s_6$ below:
  \begin{equation}
    \begin{array}{lll}
      s_1 = (U, U, F, D, D) & s_2 = (U, D, U, F, D) & s_3 = (U, F, F, F, D) \\
      s_4 = (U, F, D, U, D) & s_5 = (U, U, D, F, D) & s_6 = (U, F, U, D, D)
    \end{array}
    \label{eq: Riordan sequences}
  \end{equation}

Now we turn to the left-hand side of \eqref{eq: ConjN=1reducedform}. We have:
\begin{proposition}
  \label{prop: LHS}
  The number
  \begin{equation}
    \left\langle m_{(1,1)}^n, \sum_{\mu \in \mathcal{R}_3(2n)} s_\mu \right\rangle
  \end{equation}
  is given by the number of ballot sequences of length $n$ made up of $A$s, $B$s, and $C$s such that the parity of the number of $A$s, $B$s, and $C$s is either all even (if $n$ is even) or all odd (if $n$ is odd).
\end{proposition}

\begin{proof}
  Since $m_{(1,1)}=s_{(1,1)}$, we again need to compute Littlewood--Richardson coefficients using the dual version of Pieri's rule; however, this time the final tableau can have any shape that is a partition of $2n$ with at most three even parts. There are three possibilities at each step: we can either add a box at the end of rows one and two, rows one and three, or rows two and three. We denote these three possibilities by $A$, $B$, and $C$, respectively, so that, for instance, $A$ represents adding to rows one and two. Notice that, at all times, row one is at least as long as row two, which is at least as long as row three. This forces the corresponding sequence of $A$'s, $B$'s and $C$'s to be a ballot sequence. For example, at each step row one has $\#A + \#B$ boxes and row two has $\#A + \#C$ boxes which forces $\#B \geq \#C$ at each step. The fact that the final three-row shape is required to have all even parts forces the parities of the numbers of $A$'s, $B$'s, and $C$'s to agree.
\end{proof}

For example, when $n=6$, the ballot sequences $(A, B, A, C, B, C)$ and $(A,B,A,A,B,A)$ correspond to the following tableaux, respectively:
\begin{equation}
  \begin{ytableau}
    1 & 3 & 5 & 9 \\
    2 & 6 & 7 & 11 \\
    4 & 8 & 10 & 12
  \end{ytableau},
  \qquad
  \begin{ytableau}
    1 & 3 & 5 & 7 & 9 & 11 \\
    2 & 6 & 8 & 12 \\
    4 & 10
  \end{ytableau}.
  \label{eq: ballot tableau}
\end{equation}

The final step in the proof of Conjecture~\ref{conj:N=1 version} for $\lambda=(1^n)$ is to show these ballot sequences with equal parity are also counted by the Riordan numbers. This is indeed the case:
\begin{theorem}
  \label{thm: Riordan=ballotequalparity}
  The number of three-candidate ballot sequences of length $n$ with matching parity equals the $n$-th Riordan number $R(n)$.
\end{theorem}

\begin{proof}
  Let $r(n)$ denote the number of three-candidate ballot sequences of length $n$ with matching parity. We will show that $r(n) = R(n)$.
  Let $b_1 b_2 \dots b_n$ with $b_j \in \{1, 2, 3\}$ be a three-candidate ballot sequence of length $n$ where we use $1,2,3$ instead of $A,B,C$ to represent the three candidates. Let $m_i$ be the number of $i$'s in the sequence. There are two possibilities:
  \begin{itemize}
    \item $b_1 b_2 \dots b_n$ has matching parity, that is, $m_1 \equiv m_2 \equiv m_3 \pmod{2}$.
    \item $b_1 b_2 \dots b_n$ does not have matching parity. Let $b \in \{1, 2, 3\}$ be the candidate whose number of votes differs in parity from the other two candidates. We claim $b_1 b_2 \dots b_n b$ is a ballot sequence of length $n+1$ with matching parity. The parity condition is clear, as is the ballot sequence condition if $b = 1$. On the other hand, suppose that $b \in \{2, 3\}$. Since $m_b \not\equiv m_{b-1} \pmod{2}$, we necessarily have $m_b < m_{b-1}$, so adding the $b$ at the end of the sequence preserves the ballot sequence property.
  \end{itemize}
  The total number of three-candidate ballot sequences of length $n$ is $M(n)$. We have expressed this set as a disjoint union of those with equal parity (size $r(n)$) and those without, and given a bijection between those without and a set of size $r(n+1)$. This shows that
  \begin{equation}
    M(n) = r(n) + r(n+1).
  \end{equation}
  Note that this relationship, together with $r(1) = 0$, determines the numbers $r(n)$ uniquely. It therefore follows from $R(1)=0$ and \eqref{eq: M=R+R} that $r(n)$ equals the Riordan numbers $R(n)$.
\end{proof}

\begin{remark}
  \label{remark:EqualityofTwoSumsOfSpechtmodules}
  As we observed earlier, there are several bijections in the literature between Motzkin paths and standard Young tableaux with at most three parts. Restricting these bijections to the Riordan paths gives a scattering of tableaux of all different shapes. Our result, however, says they are equinumerous with standard Young tableaux of at most three parts, all of the same parity.
\end{remark}

Combining the previous results and Corollary~\ref{cor: Riordan sum}, we obtain the following equality between sums of character degrees.
Here, we write $(\lambda_1, \lambda_2, \lambda_3)\vdash n$ for partitions of $n$ into at most three parts (thus allowing, for instance, $\lambda_3=0$).

\begin{theorem}
  \label{thm:equalitySpecht}
  Let $X = \{ (\lambda_1, \lambda_2, \lambda_3) \vdash n \mid \lambda_1 \equiv \lambda_2 \equiv \lambda_3 \pmod{2} \}$ and let $Y = \{ (k,k,1^{n-2k}) \mid 1 \leq k \leq \lfloor n/2 \rfloor \}$. Then:
  \begin{equation}
    \sum_{\lambda \in X} f^\lambda = \sum_{\mu \in Y} f^\mu.
    \label{eq: SumsSpechtequal}
  \end{equation}
\end{theorem}

\begin{proof}
    Recall that (three-candidate) ballot sequences of length $n$ encode standard Young tableaux of size $n$ (that have at most three rows) in the following standard way: vote $i$ is cast for the $j$th candidate if $i$ appears in the $j$th row of the tableau. Under this standard bijection, the lengths $\lambda_1, \lambda_2, \lambda_3$ of the three rows correspond to the number of votes cast for the three candidates. In particular, the ballot sequence has matching parity if and only if $\lambda_1 \equiv \lambda_2 \equiv \lambda_3 \pmod{2}$. It follows that the left-hand side of \eqref{eq: SumsSpechtequal} counts the number of three-candidate ballot sequences of length $n$ with matching parity which, by Theorem~\ref{thm: Riordan=ballotequalparity}, equals the $n$-th Riordan number $R(n)$. By Corollary~\ref{cor: Riordan sum}, this matches the count for standard Young tableaux on the right-hand side of \eqref{eq: SumsSpechtequal}.
\end{proof}

\section{A proof of Conjecture~\ref{conj:N=1 version} for \texorpdfstring{$\lambda = (c^d)$}{lambda = (c,c,\ldots)}}
\label{sec: c>1 case}
In this section we prove Conjecture~\ref{conj:N=1 version} for $\lambda=(c^d)$ with $c>1$. Interestingly, when $c>1$ the values of the character sums do not depend on $c$. Rather than getting Riordan numbers, we encounter the central trinomial coefficients $T(n)=T(n,n)$, defined as the  largest coefficient of $(1+x+x^2)^n$ and represented by sequence A002426 in \cite{oeis}. We note that the more general trinomial coefficients $T(n,k)$, defined as the coefficient of $x^k$ in the expansion of $(1+x+x^2)^n$, are related to the Riordan numbers by the equation 
\begin{equation}
    \label{eq: relate trinomial and Riordan}
    R(n)= T(n,n)-T(n,n-1).
\end{equation}

As in the $c=1$ case, we compute the two sides of the conjecture separately.
So suppose that $\lambda = (c^d)$ is a partition of size $n = c d$ and denote the
right-hand side of Conjecture~\ref{conj:N=1 version} by
\begin{equation*}
  A_c (d) = \sum_{\tilde{\lambda} \in \operatorname{Ev} (\lambda)} \sum_{\mu \in
   R_2^c (2 n)} {\chi_{\tilde{\lambda}}^{\mu}}  = \sum_{\tilde{\lambda} \in
   \operatorname{Ev} (\lambda)} {\chi_{\tilde{\lambda}}^{(n, n)}}  .
\end{equation*}
We begin by showing that $A_c (d)$ is given by either the Riordan numbers $R
(d)$, if $c = 1$, or by the central trinomial coefficients $T (d)$, if $c >
1$.

\begin{theorem}
\label{thm:RHSc>1}
  We have
  \begin{equation*}
    A_c (d) = 2^d \cdot \left\{\begin{array}{ll}
       R (d), & \text{if $c = 1$,}\\
       T (d), & \text{if $c > 1$.}
     \end{array}\right.
  \end{equation*}
\end{theorem}

\begin{proof}
  As employed in \cite{rrz-char}, it follows from the fact that
  $\chi_{\lambda}^{\mu} = \langle p_{\lambda}, s_{\mu} \rangle$ that the
  values of the character $\chi^{\mu}$ can be expressed as the constant term
  \begin{equation}
    \chi_{\lambda}^{\mu} = \operatorname{ct} \left[ \frac{\prod_{1 \leq i < j
    \leq m} \left(1 - \frac{x_j}{x_i} \right) \prod_{j = 1}^r \sum_{i =
    1}^m x_i^{\lambda_j}}{\prod_{i = 1}^m x_i^{\mu_i}} \right]
    \label{eq:Sn:char:ct}
  \end{equation}
  where $m = \ell (\mu)$ and $r = \ell (\lambda)$. In our present case we have $\mu=(n,n)$ so $m=2$. We
  therefore obtain
  \begin{equation*}
    A_c (d) = \sum_{\tilde{\lambda} \in \operatorname{Ev} (\lambda)}
     {\chi_{\tilde{\lambda}}^{(n, n)}}  = \operatorname{ct} \left[ \frac{\left(1 -
     \frac{x_2}{x_1} \right)}{(x_1 x_2)^n} \sum_{\tilde{\lambda} \in \operatorname{Ev}
     (\lambda)} \prod_{j = 1}^{\ell (\tilde{\lambda})}
     (x_1^{\tilde{\lambda}_j} + x_2^{\tilde{\lambda}_j}) \right] .
  \end{equation*}
  Observe that, for $\lambda = (c^d)$, the multiset $\operatorname{Ev} (\lambda)$
  consists of the partitions
  \begin{equation*}
    \left((2 c)^k {, c^{2 (d - k)}}  \right)
  \end{equation*}
  with multiplicity $\binom{d}{k}$ and with $k \in \{ 0, 1, \ldots, d \}$.
  Consequently,
  \begin{eqnarray*}
    A_c (d) & = & \operatorname{ct} \left[ \frac{\left(1 - \frac{x_2}{x_1}
    \right)}{(x_1 x_2)^n} \sum_{k = 0}^d \binom{d}{k} (x_1^{2 c} + x_2^{2
    c})^k (x_1^c + x_2^c)^{2 (d - k)} \right]\\
    & = & \operatorname{ct} \left[ \frac{\left(1 - \frac{x_2}{x_1} \right)}{(x_1
    x_2)^n} ((x_1^{2 c} + x_2^{2 c}) + (x_1^c + x_2^c)^2)^d \right]\\
    & = & 2^d \operatorname{ct} \left[ \frac{\left(1 - \frac{x_2}{x_1} \right)}{(x_1
    x_2)^n} (x_1^{2 c} + x_1^c x_2^c + x_2^{2 c})^d \right] .
  \end{eqnarray*}
  Since $n = c d$, the latter simplifies to
  \begin{eqnarray*}
    A_c (d) & = & 2^d \operatorname{ct} \left[ \left(1 - \frac{x_2}{x_1} \right)
    \left(\left(\frac{x_1}{x_2} \right)^c + 1 + \left(\frac{x_2}{x_1}
    \right)^c \right)^d \right]\\
    & = & 2^d \operatorname{ct} \left[ (1 - x) \left(\frac{1}{x^c} + 1 + x^c
    \right)^d \right] .
  \end{eqnarray*}
  If $c = 1$, then the claim follows from the well-known representation
  \begin{equation*}
    R (d) = \operatorname{ct} \left[ (1 - x) \left(\frac{1}{x} + 1 + x \right)^d
     \right]
  \end{equation*}
  of the Riordan numbers. On the other hand, suppose that $c > 1$. Then the
  expansion of $(x^{- c} + 1 + x^c)^d$ only features terms with exponents that
  are multiples of $c$. Therefore,
  \begin{equation*}
    \operatorname{ct} \left[ (1 - x) \left(\frac{1}{x^c} + 1 + x^c \right)^d \right]
     = \operatorname{ct} \left[ \left(\frac{1}{x^c} + 1 + x^c \right)^d \right] =
     \operatorname{ct} \left[ \left(\frac{1}{x} + 1 + x \right)^d \right]
  \end{equation*}
  is equal to the central trinomial coefficient $T (d)$.
\end{proof}

Likewise, for a partition $\lambda = (c^d)$ of size $n = c d$, denote the
left-hand side of Conjecture~\ref{conj:N=1 version} by
\begin{equation*}
  B_c (d) = \sum_{\tilde{\lambda} \in \operatorname{Ev} (\lambda)} \sum_{\mu \in R_3
   (2 n)} (- 1)^{\ell (\tilde{\lambda})} {\chi_{\tilde{\lambda}}^{\mu}}  .
\end{equation*}
In Theorem \ref{thm: Riordan=ballotequalparity} we already showed that $B_1 (d) = 2^d R (d)$. The following, combined
with the previous theorem, therefore shows that Conjecture~\ref{conj:N=1
version} is true for all partitions with a single part size.

\begin{theorem}
  If $c > 1$, then we have
  \begin{equation*}
    B_c (d) = 2^d T (d) .
  \end{equation*}
\end{theorem}

\begin{proof}
  We again begin by expressing the characters $\chi^{\mu}$ in terms of the
  constant terms \eqref{eq:Sn:char:ct}. In the present case, this leads to
  \begin{equation*}
    B_c (d) = \sum_{\tilde{\lambda} \in \operatorname{Ev} (\lambda)} \sum_{(\mu_1,
     \mu_2, \mu_3)} (- 1)^{\ell (\tilde{\lambda})} \operatorname{ct} \left[
     \frac{\prod_{1 \leq i < j \leq 3} \left(1 - \frac{x_j}{x_i}
     \right) \prod_{j = 1}^{\ell (\tilde{\lambda})} \sum_{i = 1}^3
     x_i^{\tilde{\lambda}_j}}{\prod_{i = 1}^3 x_i^{2 \mu_i}} \right]
  \end{equation*}
  where the inner sum is over (weak) partitions of $n$ into three parts: that
  is, $(\mu_1, \mu_2, \mu_3)$ with $\mu_1 \geq \mu_2 \geq \mu_3
  \geq 0$ and $\mu_1 + \mu_2 + \mu_3 = n$. This is possible since the
  formula \eqref{eq:Sn:char:ct} gives the same value if $\mu = (\mu_1, \ldots,
  \mu_m)$ is replaced by $\mu = (\mu_1, \ldots, \mu_m, 0)$ because the extra
  variable $x_{m + 1}$ only appears with nonnegative exponents and so cannot
  contribute to the constant term.
  
  As in the previous proof, we use that, for $\lambda = (c^d)$, the multiset
  $\operatorname{Ev} (\lambda)$ consists of the partitions $\left((2 c)^k {, c^{2 (d
  - k)}}  \right)$ with multiplicity $\binom{d}{k}$ and with $k \in \{ 0, 1,
  \ldots, d \}$. We therefore find, for any integer $m \geq 1$,
  \begin{eqnarray*}
    \sum_{\tilde{\lambda} \in \operatorname{Ev} (\lambda)} (- 1)^{\ell
    (\tilde{\lambda})} \prod_{j = 1}^{\ell (\tilde{\lambda})} \sum_{i = 1}^m
    x_i^{\tilde{\lambda}_j} & = & \sum_{k = 0}^d \binom{d}{k} (- 1)^k \left(\sum_{i = 1}^m x_i^{2 c} \right)^k \left(\sum_{i = 1}^m x_i^c \right)^{2
    (d - k)}\\
    & = & \left(\left(- \sum_{i = 1}^m x_i^{2 c} \right) + \left(\sum_{i =
    1}^m x_i^c \right)^2 \right)^d\\
    & = & 2^d m_{1, 1} (x_1^c, x_2^c, \ldots, x_m^c) .
  \end{eqnarray*}
  Applied to our situation, this implies that
  \begin{equation}
    2^{- d} B_c (d) = \operatorname{ct} \left[ (x_1^c x_2^c + x_2^c x_3^c + x_3^c
    x_1^c)^d \sum_{(\mu_1, \mu_2, \mu_3)} \frac{\left(1 - \frac{x_2}{x_1}
    \right) \left(1 - \frac{x_3}{x_1} \right) \left(1 - \frac{x_3}{x_2}
    \right)}{x_1^{2 \mu_1} x_2^{2 \mu_2} x_3^{2 \mu_3}} \right] .
    \label{eq:B:ct:sum:mu}
  \end{equation}
  The sum
  \begin{equation*}
    S = \sum_{(\mu_1, \mu_2, \mu_3)} \frac{\left(1 - \frac{x_2}{x_1} \right)
     \left(1 - \frac{x_3}{x_1} \right) \left(1 - \frac{x_3}{x_2}
     \right)}{x_1^{2 \mu_1} x_2^{2 \mu_2} x_3^{2 \mu_3}}
  \end{equation*}
  over partitions $(\mu_1, \mu_2, \mu_3)$ of $n$ into three parts is a
  Laurent polynomial in $x_1, x_2, x_3$. Of the monomials in that Laurent
  polynomial, only few contribute to the constant term and there is
  considerable cancellation among those that contribute. To describe this, we
  expand
  \begin{equation*}
    \left(1 - \frac{x_2}{x_1} \right) \left(1 - \frac{x_3}{x_1} \right)
     \left(1 - \frac{x_3}{x_2} \right) = \alpha + \beta + \gamma
  \end{equation*}
  where
  \begin{equation*}
    \alpha = 1 - \frac{x_3^2}{x_1^2}, \quad \beta = \frac{x_2 x_3}{x_1^2} -
     \frac{x_3}{x_2}, \quad \gamma = \frac{x_3^2}{x_1 x_2} - \frac{x_2}{x_1} .
  \end{equation*}
  \textbf{Case $\alpha$. }First, we consider the monomials in the sum $S$
  that arise from $\alpha$. These are
  \begin{equation*}
    \sum_{(\mu_1, \mu_2, \mu_3)} \frac{1 - x_1^{- 2} x_3^2}{x_1^{2 \mu_1}
     x_2^{2 \mu_2} x_3^{2 \mu_3}} = \sum_{\substack{
       (\mu_1, \mu_2, \mu_3)\\
       \mu_1 = \mu_2 \text{ or } \mu_2 = \mu_3
     }} \frac{1}{x_1^{2 \mu_1} x_2^{2 \mu_2} x_3^{2 \mu_3}} -
     \sum_{\substack{
       (\mu_1, \mu_2, \mu_3)\\
       \mu_3 = 0
     }} \frac{x_1^{- 2} x_3^2}{x_1^{2 \mu_1} x_2^{2 \mu_2} x_3^{2
     \mu_3}}
  \end{equation*}
  where most of the terms on the left-hand side cancelled in pairs. Observe
  that the monomials from the final sum do not contribute to the constant term
  \eqref{eq:B:ct:sum:mu} because $x_3$ appears with a positive exponent
  (namely as $x_3^2$). On the other hand, the first sum on the right-hand side
  splits into
  \begin{equation*}
    \sum_{\substack{
       (\mu_1, \mu_2, \mu_3)\\
       \mu_2 = \mu_3
     }} \frac{1}{x_1^{2 \mu_1} x_2^{2 \mu_2} x_3^{2 \mu_3}} =
     \sum_{\substack{
       m = 0
     }}^{\lfloor n / 3 \rfloor} \frac{1}{(x_3 x_2)^{2 m} x_1^{2 (n
     - 2 m)}}
  \end{equation*}
  as well as
  \begin{equation*}
    \sum_{\substack{
       (\mu_1, \mu_2, \mu_3)\\
       \mu_1 = \mu_2 > \mu_3
     }} \frac{1}{x_1^{2 \mu_1} x_2^{2 \mu_2} x_3^{2 \mu_3}} =
     \sum_{\substack{
       m = \lfloor n / 3 \rfloor + 1
     }}^{\lfloor n / 2 \rfloor} \frac{1}{(x_1 x_2)^{2 m} x_3^{2 (n
     - 2 m)}} .
  \end{equation*}
  Inside the constant term \eqref{eq:B:ct:sum:mu}, these are multiplied with
  the polynomial $(x_1^c x_2^c + x_2^c x_3^c + x_3^c x_1^c)^d$. Because the
  latter is symmetric, we can permute the variables $x_1, x_2, x_3$ in each
  monomial above. Overall, we therefore conclude that the contribution to the
  constant term \eqref{eq:B:ct:sum:mu} by monomials in $S$ arising from
  $\alpha$ is
  \begin{eqnarray}
    &  & \operatorname{ct} \left[ (x_1^c x_2^c + x_2^c x_3^c + x_3^c x_1^c)^d
    \sum_{\substack{
      m = 0
    }}^{\lfloor n / 2 \rfloor} \frac{1}{(x_1 x_2)^{2 m} x_3^{2 (n -
    2 m)}} \right] \nonumber\\
    & = & \operatorname{ct} \left[ (x_1^c + x_2^c + x_1^c x_2^c)^d
    \sum_{\substack{
      m = 0
    }}^{\lfloor n / 2 \rfloor} \frac{1}{(x_1 x_2)^{2 m}} \right] . 
    \label{eq:ct:contrib:a}
  \end{eqnarray}
  For the equality, we used that all terms are homogeneous, allowing us to set
  $x_3 = 1$ without changing the constant term.
  
  \textbf{Case $\beta$. }Next, we similarly consider the monomials in the
  sum $S$ that arise from $\beta$. These are
  \begin{equation*}
    \sum_{(\mu_1, \mu_2, \mu_3)} \frac{x_1^{- 2} x_2 x_3 - x_2^{- 1}
     x_3}{x_1^{2 \mu_1} x_2^{2 \mu_2} x_3^{2 \mu_3}} =
     \sum_{\substack{
       (\mu_1, \mu_2, \mu_3)\\
       \mu_2 = \mu_3
     }} \frac{x_1^{- 2} x_2 x_3}{x_1^{2 \mu_1} x_2^{2 \mu_2} x_3^{2
     \mu_3}} - \sum_{\substack{
       (\mu_1, \mu_2, \mu_3)\\
       \mu_1 \leq \mu_2 + 1
     }} \frac{x_2^{- 1} x_3}{x_1^{2 \mu_1} x_2^{2 \mu_2} x_3^{2
     \mu_3}}
  \end{equation*}
  where we again cancelled most of the terms on the left-hand side. Note that
  the monomials in the final sum are of the form $x_1^{- 2 \mu_1} x_2^{- 2
  \mu_2 - 1} x_3^{- 2 \mu_3 + 1}$. If $\mu_1 = \mu_2$ this is $x_1^{- 2 \mu_1}
  x_2^{- 2 \mu_1 - 1} x_3^{- 2 \mu_3 + 1}$ and if $\mu_1 = \mu_2 + 1$ this is
  $x_1^{- 2 \mu_1} x_2^{- 2 \mu_1 + 1} x_3^{- 2 \mu_3 + 1}$; in either case,
  the exponents of $x_1$ and $x_2$ differ by exactly $1$. As such, they cannot
  both be divisible by $c > 1$ and so the monomials cannot contribute to the
  constant term \eqref{eq:B:ct:sum:mu}. We rewrite the other sum as
  \begin{equation*}
    \sum_{\substack{
       (\mu_1, \mu_2, \mu_3)\\
       \mu_2 = \mu_3
     }} \frac{x_1^{- 2} x_2 x_3}{x_1^{2 \mu_1} x_2^{2 \mu_2} x_3^{2
     \mu_3}} = \sum_{\substack{
       m = 0
     }}^{\lfloor n / 3 \rfloor} \frac{1}{(x_3 x_2)^{2 m - 1} x_1^{2
     (n - 2 m + 1)}} .
  \end{equation*}
  Note that the term corresponding to $m = 0$ does not contribute to the
  constant term \eqref{eq:B:ct:sum:mu}. Similar to the case $\alpha$, we swap
  $x_1$ and $x_3$ in these monomials, then set $x_3 = 1$, to find that the
  contribution to the constant term \eqref{eq:B:ct:sum:mu} by monomials in $S$
  arising from $\beta$ is
  \begin{equation}
    \operatorname{ct} \left[ (x_1^c + x_2^c + x_1^c x_2^c)^d
    \sum_{\substack{
      m = 1
    }}^{\lfloor n / 3 \rfloor} \frac{1}{(x_1 x_2)^{2 m - 1}}
    \right] . \label{eq:ct:contrib:b}
  \end{equation}
  \textbf{Case $\gamma$. }Finally, the monomials in the sum $S$ that arise
  from $\gamma$ are
  \begin{equation*}
    \sum_{(\mu_1, \mu_2, \mu_3)} \frac{x_1^{- 1} x_2^{- 1} x_3^2 - x_1^{- 1}
     x_2}{x_1^{2 \mu_1} x_2^{2 \mu_2} x_3^{2 \mu_3}} =
     \sum_{\substack{
       (\mu_1, \mu_2, \mu_3)\\
       \mu_1 = \mu_2 \text{ or } \mu_3 = 0
     }} \frac{x_1^{- 1} x_2^{- 1} x_3^2}{x_1^{2 \mu_1} x_2^{2
     \mu_2} x_3^{2 \mu_3}} - \sum_{\substack{
       (\mu_1, \mu_2, \mu_3)\\
       \mu_2 \leq \mu_3 + 1
     }} \frac{x_1^{- 1} x_2}{x_1^{2 \mu_1} x_2^{2 \mu_2} x_3^{2
     \mu_3}} .
  \end{equation*}
  As in the case $\beta$, the final sum does not contribute to the constant
  term \eqref{eq:B:ct:sum:mu} when $c > 1$ because the exponents of $x_2$ and
  $x_3$ differ by exactly $1$. Further, as in the case $\alpha$, the monomials
  corresponding to $\mu_3 = 0$ in the first sum on the right-hand side do not
  contribute to the constant term \eqref{eq:B:ct:sum:mu} because $x_3$ appears
  with a positive exponent. We rewrite the remaining terms as
  \begin{equation*}
    \sum_{\substack{
       (\mu_1, \mu_2, \mu_3)\\
       \mu_1 = \mu_2
     }} \frac{x_1^{- 1} x_2^{- 1} x_3^2}{x_1^{2 \mu_1} x_2^{2
     \mu_2} x_3^{2 \mu_3}} = \sum_{\substack{
       m = \lceil n / 3 \rceil
     }}^{\lfloor n / 2 \rfloor} \frac{1}{(x_1 x_2)^{2 m + 1} x_3^{2
     (n - 2 m - 1)}} .
  \end{equation*}
  Note that the term corresponding to $m = \lfloor n / 2 \rfloor$ does not
  contribute to the constant term \eqref{eq:B:ct:sum:mu} if $\lfloor n / 2
  \rfloor > \lfloor (n - 1) / 2 \rfloor$ because the exponent of $x_3$ is
  positive in that case. Setting $x_3 = 1$, we thus record that the
  contribution to the constant term \eqref{eq:B:ct:sum:mu} by monomials in $S$
  arising from $\gamma$ is
  \begin{equation}
    \operatorname{ct} \left[ (x_1^c + x_2^c + x_1^c x_2^c)^d
    \sum_{\substack{
      m = \lceil n / 3 \rceil
    }}^{\lfloor (n - 1) / 2 \rfloor} \frac{1}{(x_1 x_2)^{2 m + 1}}
    \right] . \label{eq:ct:contrib:c}
  \end{equation}
  We now claim that the combined contribution to the constant term
  \eqref{eq:B:ct:sum:mu} by monomials in $S$ arising from $\beta$ and $\gamma$
  is
  \begin{equation}
    \operatorname{ct} \left[ (x_1^c + x_2^c + x_1^c x_2^c)^d
    \sum_{\substack{
      m = 0
    }}^{\lfloor (n - 1) / 2 \rfloor} \frac{1}{(x_1 x_2)^{2 m + 1}}
    \right] . \label{eq:ct:contrib:bc}
  \end{equation}
  This almost follows by simply summing the individual sums
  \eqref{eq:ct:contrib:b} and \eqref{eq:ct:contrib:c}, except that the
  combined sum \eqref{eq:ct:contrib:bc} has one additional term if $n \equiv
  1, 2$ modulo $3$; namely, the term corresponding to $m = \lfloor n / 3
  \rfloor$. We need to show that this term does not contribute to the constant
  term \eqref{eq:B:ct:sum:mu}. To see this, write $n = 3 r + \nu$ for $\nu \in
  \{ 1, 2 \}$ so that the extra term corresponds to $m = r$ and that the
  exponents of $x_1$ and $x_2$ are $- (2 r + 1)$. We readily confirm that
  $\gcd (n, 2 r + 1) = \gcd (3 r + \nu, 2 r + 1) = 1$ which implies that $2 r
  + 1$ cannot be a multiple of $c$. In particular, the extra term cannot
  contribute to the constant term \eqref{eq:B:ct:sum:mu} if $c > 1$.
  
  Finally, we conclude that \eqref{eq:B:ct:sum:mu} equals the sum of
  \eqref{eq:ct:contrib:a} and \eqref{eq:ct:contrib:bc}, resulting in
  \begin{equation*}
    2^{- d} B_c (d) = \operatorname{ct} \left[ (x_1^c + x_2^c + x_1^c x_2^c)^d
     \sum_{\substack{
       m = 0
     }}^n \frac{1}{(x_1 x_2)^m} \right] = \operatorname{ct} \left[ \left(x^c + \frac{1}{x^c y^c} + \frac{1}{y^c} \right)^d
     \sum_{\substack{
       m = 0
     }}^n y^m \right]
  \end{equation*}
  where we substituted $x = x_1$ and $y = 1 / (x_1 x_2)$, so that $x_2 = 1 /
  (x y)$, to obtain the latter constant term. Since the sum in that constant
  term is over all possible powers of $y$ that can contribute to the constant
  term, we obtain the overall constant term by setting $y = 1$. Hence,
  \begin{equation*}
    2^{- d} B_c (d) = \operatorname{ct} \left[ \left(x^c + \frac{1}{x^c} + 1
     \right)^d \right] = \operatorname{ct} \left[ \left(x + \frac{1}{x} + 1 \right)^d
     \right] = T (d),
  \end{equation*}
  as claimed.
\end{proof}

\begin{corollary}
  For integers $c > 1$ and $d \ge 1$ we have
  $$\left\langle m_{(c,c)}^d, \sum_{\mu \in \mathcal{R}_{3}(2cd)} s_\mu \right\rangle = T(d).$$
\end{corollary}

\section{Conclusions and future work}

We have proved Conjecture~\ref{conj:N=1 version} for all partitions with a single part size, thus affirmatively answering Question~\ref{q:amdeberhan:x} for those partitions and $N = 1$. It is a natural question to pursue whether the present techniques can be extended to prove identity \eqref{eq: N=1form} of Conjecture~\ref{conj:N=1 version} for all partitions. When summing \eqref{eq: N=1form} over all partitions, we get Amdeberhan's identity \eqref{eq: q:amdeberhan} specialized to $N=1$. That sum appears to admit the following simple closed formula:

\begin{conjecture}For any integer $n\ge1$,
  \begin{equation}
    \sum_{\lambda \vdash n} \frac{1}{z_\lambda} \sum_{\tilde{\lambda} \in \Ev} \sum_{\mu \in \mathcal{R}_3(2n)} (-1)^{\ell(\tilde{\lambda})} \chi^\mu_{\tilde{\lambda}} =
    \sum_{\lambda \vdash n} \frac{1}{z_\lambda} \sum_{\tilde{\lambda} \in \Ev} \chi^{(n,n)}_{\tilde{\lambda}} =
    \left\{\begin{array}{ll}
       \binom{\frac{n}{2}+2}{2}, & \text{if $n$ is even,}\\
       0, & \text{if $n$ is odd.}
     \end{array}\right.    
  \end{equation}
\end{conjecture}

In private communication, Amdeberhan has shared that his Question~\ref{q:amdeberhan} is inspired by the following conjectured equality of $q$-series. Here, given a partition $\lambda = (\lambda_1, \lambda_2, \ldots, \lambda_r)$, we denote
\[ g_{\lambda} (q) = \prod_{j = 1}^r \frac{q^{\lambda_j}}{1 + q^{\lambda_j}} .
\]
\begin{conjecture}[Amdeberhan]
  \label{conj:amdeberhan:q}For all integers $N \geq 1$,
  \begin{equation}
    \sum_{n \geq 0} \sum_{\lambda \vdash n} \frac{g_{\lambda}
    (q)}{z_{\lambda}} \sum_{\tilde{\lambda} \in \Ev} \sum_{\mu
    \in R_{2 N + 1} (2 n)} (- 1)^{\ell (\tilde{\lambda})}
    {\chi_{\tilde{\lambda}}^{\mu}}  = \sum_{n \geq 0} \sum_{\lambda
    \vdash n} \frac{g_{\lambda} (q)}{z_{\lambda}} \sum_{\tilde{\lambda} \in
    \Ev} \sum_{\mu \in R_{2 N}^c (2 n)}
    {\chi_{\tilde{\lambda}}^{\mu}}  . \label{eq:amdeberhan:q}
  \end{equation}
\end{conjecture}

Note that $g_{\lambda} (q) = O (q^{| \lambda |})$, allowing us to verify
equation~\eqref{eq:amdeberhan:q} for fixed $N$ up to terms of order $q^m$ by
truncating the outer sums to $n \leqslant m$. Doing so, we find, for instance,
that for $N = 1$ both sides of \eqref{eq:amdeberhan:q} equal
\[ 1 + 3 q^2 - 4 q^3 + 9 q^4 - 12 q^5 + 22 q^6 - 36 q^7 + 60 q^8 - 88 q^9 +
   135 q^{10} + O (q^{11}) . \]

Various bijections are known between Motzkin paths and standard Young tableaux with at most three parts. It would be of interest to identify such a bijection with the additional property that the subset of Riordan paths is mapped to standard Young tableaux of at most three parts, all of the same parity. Theorem~\ref{thm: Riordan=ballotequalparity} shows that this is possible but does not provide an explicit bijection.

\medskip
\textbf{Acknowledgements. }We thank Tewodros Amdeberhan for kindly sharing
details on his Question~\ref{q:amdeberhan} as well as allowing us to include
his motivating Conjecture~\ref{conj:amdeberhan:q}.

\appendix

\section{Counterexamples to Question~\ref{q:amdeberhan:x}}
\label{sec:appendix-counterexamples}
We showed that the answer to Question~\ref{q:amdeberhan:x} is affirmative for $N=1$ and partitions with one part size, and we conjecture that the answer continues to be affirmative for $N=1$ in general.  Here, we illustrate that identity \eqref{eq: stronger form} does not, however, hold in general.

While identity \eqref{eq: stronger form} holds for all partitions of size $n\le7$, we find that it holds for partitions of size $n=8$ only if $N\neq3$. In the case $N=3$, we find that \eqref{eq: stronger form} holds for the partitions $\{(8), (7, 1), (6, 2), (6, 1, 1), (4, 2, 1, 1), (2, 2, 2, 1, 1)\}$ but not for other partitions of size $8$.

For instance, consider $\lambda=(5,2,1) \vdash 8$. Then:

\begin{equation}
\begin{aligned}
\Ev = \{&(10,4,2), (10,4,1,1), (10,2,2,2), (10,2,2,1,1), \\
&(5,5,4,2), (5,5,4,1,1), (5,5,2,2,2), (5,5,2,2,1,1)\}.
\end{aligned}
\label{eq:evforctex}
\end{equation}
One can check that the identity \eqref{eq: stronger form} in Question~\ref{q:amdeberhan:x} holds for $N=1$ and $N=2$ in this case. For $N=3$ the new partitions $\mu$ in the left-hand sum are those with exactly six or seven parts, all even, namely $(6,2^5), (4^2,2^4),(4,2^6)$.
The new partitions on the right-hand side are those with even multiplicities and exactly six parts, namely $(3^4,2^2), (4^2,2^4), (4^2,3^2,1^2),(5^2,2^2,1^2),(6^2,1^4)$.

Calculating the character sums we find that the contributions of the new partitions to the sum on the left-hand side of \eqref{eq: stronger form} is zero while they are $-8$ on the right-hand side. As a result the two sides of \eqref{eq: stronger form} differ by $8$ for $N=3$. For $N=4$, the additional new partitions contribute $-8$ on the left-hand side and $0$ on the right-hand side so that \eqref{eq: stronger form} again holds for $N=4$. Indeed, we find that \eqref{eq: stronger form} holds for $N\ge4$, as predicted in Remark \ref{remark: RN} for $N\ge8$.

For larger $n$, the discrepancies between the two sides of \eqref{eq: stronger form} can get more pronounced, although for $N=1$ and $N=2$ we have not observed any partitions $\lambda$ for which \eqref{eq: stronger form} does not hold. For example, when $\lambda=(3^2,2^3,1)$ we find that the right-hand side of \eqref{eq: stronger form} exceeds the left-hand side by $5184$ for $N=3$, by $7488$ for $N=4$, and by $2368$ for $N=5$.  For other values of $N$, the identity \eqref{eq: stronger form} does hold for that partition $\lambda$.

\bibliographystyle{alpha}
\bibliography{SymmetricFunctionsReferences}

\end{document}